\documentclass[11pt]{amsart}
\usepackage{amsmath}
\usepackage{amsfonts}
\usepackage{amsthm}
\usepackage{amssymb}
\usepackage{mathrsfs}
\usepackage{latexsym}
\usepackage{verbatim}
\usepackage{diagrams}
\usepackage{dsfont}
\usepackage{stmaryrd}
\usepackage{graphics}

\swapnumbers

\newtheorem{theorem}{Theorem}[section]

\newtheorem{df}[theorem]{Definition}
\newtheorem{lemma}[theorem]{Lemma}

\def\rec{\varsigma}
\def\OL{\mathcal{O}}
\def\varrhobar{\overline{\varrho}}
\def\Tr{\mathrm{Tr}}
\def\GN{N.S_5}
\def\HN{N.A_5}
\def\PP{\mathbf{P}}
\def\ww{\omega}
\def\xxi{\phi}

\def\R{\mathbf{R}}
\def\Qbar{\overline{\Q}}

\def\p{\mathfrak{p}}

\def\As{\mathrm{Asai}}
\def\PGL{\mathrm{PGL}}

\def\Sym{\mathrm{Sym}}
\def\SL{\mathrm{SL}}

\def\F{\mathbf{F}}
\def\Fbar{\overline{\F}}
\def\Q{\mathbf{Q}}
\def\GL{\mathrm{GL}}
\def\Z{\mathbf{Z}}
\def\Gal{\mathrm{Gal}}
\def\C{\mathbf{C}}
\def\Re{\mathrm{Re}}
\def\Frob{\mathrm{Frob}}

\def\K{K^{\mathrm{gal}}}
\def\L{\widetilde{\K}}

\begin{document}

\title{The Artin conjecture for some $S_5$-extensions}
\author{Frank Calegari}
\thanks{Supported in part by 
NSF Career Grant DMS-0846285 
and the Sloan Foundation. MSC2010 classification: 11F66,  11F80, 11R39, 11S37}
\maketitle

\section{Introduction}

Let $G_{\Q}$ denote the absolute Galois group of $\Q$, let $K/\Q$ be a number field
with Galois closure $\K$, and let
$$\rho: G_{\Q} \rightarrow \Gal(\K/\Q) \hookrightarrow \GL_n(\C)$$
be a continuous irreducible 
Galois representation. Attached to $\rho$ is an $L$-function $L(s,\rho)$ which is holomorphic
for $\Re(s) > 1$. Artin conjectured that $L(\rho,s)$ had a holomorphic  continuation
to the entire complex plane, with the possible exception of a pole at $s = 1$ if $\rho$ was
the trivial representation. Subsequently, Langlands conjectured that $\rho$ was 
\emph{modular}, that is, there exists   an  automorphic representation
$\pi$ for $\GL(n)/\Q$ such that $L(\pi,s) = L(\rho,s)$ and $\pi$ is cuspidal as long as $\rho$ is non-trivial, which in particular implies Artin's
conjecture.  If $\rho$ is a monomial representation (that is, induced from a character) then
Artin's conjecture was proved by Artin, and if $G$ is nilpotent, then Langlands' conjecture is a consequence of 
cyclic base change (Theorem 7.1 of~\cite{AC}). Suppose that  $G = \Gal(\K/\Q)$ has a faithful permutation representation
of degree $\le 5$. If $G$ is solvable, then Artin's conjecture follows from
the fact that all such $G$ are monomial. Moreover, Langlands' conjectures  are also known in these cases.
 If $G = A_4$ and $G = S_4$, then  results of Langlands~\cite{GL} and Tunnell~\cite{Tunnell} imply
that there exists a cuspidal representation $\pi$ for $\GL(2)/\Q$ which is associated to a
representation with \emph{projective} image $G$. The faithful representations of $G$ (which are all of dimension
three) can then be realized (up to twist) by the symmetric square $\Sym^2 \pi$ of Gelbart and Jacquet. 
If $G$ is \emph{not} solvable, then either $G$ is $A_5$ or $S_5$.
If $G = A_5$, then Langlands' conjecture is known providing that complex conjugation is non-trivial.
A theorem of Khare--Wintenberger~\cite{Khare} guarantees the existence of an automorphic form $\pi$ for
$\GL(2)/\Q$ corresponding to any odd projective $A_5$-representation, and then all the faithful
representations of $A_5$ can be reconstructed from $\pi$ (and $\iota \pi$ for the outer automorphism
$\iota$ of $A_5$) via functoriality (see~\cite{KI}).
In this note, we prove some
cases of
the conjectures of Artin and Langlands for $S_5$ using functoriality, in the spirit of 
Tunnell~\cite{Tunnell}. Recall that the faithful irreducible representations of $S_5$ have dimensions $4$, $5$, and $6$.

\begin{theorem} \label{theorem:main} Let $K/\Q$ be an extension
such that $\Gal(\K/\Q) = S_5$. Suppose, furthermore, that
\begin{enumerate}
\item Complex conjugation in $\Gal(\K/\Q) = S_5$  has conjugacy class $(12)(34)$.
\item The extension $\K/\Q$ is unramified at $5$, and the Frobenius element $\Frob_5 \in \Gal(\K/\Q) = S_5$  has conjugacy class $(12)(34)$.
 \end{enumerate}
If $\rho$ is irreducible of dimension $4$ or $6$, then $\rho$ is modular.
 If $\rho$ is irreducible of dimension $5$, there exists a tempered cuspidal $\varpi$ for $\GL(5)/\Q$ such that
$\mathrm{WD}(\rho|D_v) \simeq \mathrm{rec}(\varpi_v)$ for a set of places $v$ of density one.
\end{theorem}

The existence of a weak correspondence between $\rho$ and $\varpi$, though perhaps an approximation to 
Langlands' conjecture, is not sufficient even to deduce Artin's
conjecture for $L(\rho,s)$ (although one \emph{can} deduce from our arguments that $L(\rho,s)$ is holomorphic
for $\Re(s) > 1 - c$ for some ineffective constant $c > 0$). We may, however, remedy this lacuna under
a more stringent hypothesis, which shows that the conjectures
of Artin and Langlands \emph{can} be established unconditionally for some $S_5$-extensions.

\begin{theorem} Let $K$ be as in Theorem~\ref{theorem:main}. \label{theorem:two}
Let $E/\Q$ be the quadratic subfield of $\K$, 
let $F/\Q$ be a subfield of $\K$ of degree $6$ over $\Q$, and  
let $H$ be the compositum of $F$ and $E$.
Suppose that 
$\zeta_{H}(s)$ does not vanish for real $s \in (0,1)$.
Then $\rho$ is modular.
\end{theorem}

Note that the non-vanishing condition on $\zeta_{H}(s)$ is explicitly verifiable in theory
(and in practice, see Theorem~\ref{example:booker}).
There is a factorization 
$$\zeta_{H}(s) =  \zeta_F(s)  L(\eta,E/\Q,s)  L(\rho_5,s)$$
for some explicit meromorphic Artin $L$-function $L(\rho_5,s)$.
Our result actually only requires the non-vanishing of $L(\rho_5,s)$
for $s \in (0,1)$.  
 The non-vanishing of Artin $L$-functions for $s \in (0,1)$ is not implied by the GRH, and
 indeed Armitage found examples of number fields $L$ such that $\zeta_{L}(\frac{1}{2}) = 0$
 (see~\cite{Arm}).
 Those constructions, however, arose from Artin $L$-functions $\rho$ with real traces
 such that the global Artin root number $W(\rho)$ was $-1$.
  Since all the  representations of $S_5$ (indeed of $S_n$) are self-dual and
 definable over $\R$, the
 the Artin  root number $W(\rho)$ of any irreducible representation of these groups is automatically $+1$, and
so $\zeta_{H}(s)$ is never forced to vanish at  $s = 1/2$ for sign reasons. 
Indeed,  the non-vanishing of $\zeta_{H}(s)$ would follow if one assumed that (in addition to
the GRH) that Artin $L$-functions of irreducible representations have only simple zeros.
 The local condition at $5$ is not essential to the method, but the restriction on complex conjugation is completely essential. 
 The reason we can prove anything
non-trivial under these assumptions is because we can reduce to known cases of Artin for (projective) two-dimensional 
$A_5$-extensions over quadratic extensions of~$\Q$. 
The assumption that complex conjugation is conjugate to $(12)(34)$ ensures that these Artin representations are defined
over (totally) real quadratic fields, and that the image of complex conjugation in $\GL_2(\C)$ has determinant $-1$.
For such two dimensional representations, one can deduce modularity using results of Sasaki~\cite{Sasaki1,Sasaki2},
which generalizes the approach of~\cite{BDST}.

\subsection{Acknowledgments} 

I would like to thank Andrew Booker for establishing the non-vanishing of $L(\rho_5,s)$ for $s \in (0,1)$ for an explicit
number field $K$ (see Theorem~\ref{theorem:booker}),  demonstrating that one
can effectively use the results of this paper to prove the Artin conjecture for particular $S_5$-extensions.
I would also like to thank Lassima Dembele for computing the Hilbert modular forms of weight $(2,2)$ and level one
for $\Q(\sqrt{1609})$. Finally, I would like to thank Kevin Buzzard,  Peter Sarnak, and Dinakar Ramakrishnan for useful conversations.

\section{Some Group Theory}

Let $\F^{\times}_5 \subset \Delta \subset \Fbar^{\times}_5$ be a finite subgroup,
necessarily cyclic.
Let $N:=|\Delta|$ denote the order of $\Delta$ --- it is divisible by $4$.
We define the group $\GN$ to be $\GL_2(\F_5) \Delta$. There is a tautological map:
$$\GN \rightarrow \PGL_2(\F_5) \simeq S_5$$
which realizes $\GN$ as a central extension of $S_5$ by the cyclic group
$\Delta$ of order $N$. Let $\HN$ denote the kernel of the composite
$$\GN \rightarrow S_5 \rightarrow \Z/2\Z.$$

\begin{lemma} The group $\HN$ is a central extension of $A_5$ by $\Delta$,
and admits a faithful complex  representation $\varrho$  of dimension two.
Any character of $\HN$ is determined by its restriction to $\Delta$. 
\end{lemma}

\begin{proof}
Since the image of $\HN$ in $\PGL_2(\F_5)$ is $A_5$ and the kernel
is $\Delta$, it is clear that $\HN$ is a central extension.
Central extensions of a group $G$ by a cyclic group $\Delta$ of
order $N$ are classified by $H^2(G,\mu_N)$. It is a result (essentially of Schur) that
$H^2(A_5,\C^{\times}) = \Z/2\Z$. 
Since $A_5$ is a perfect group, $H^1(A_5,\C^{\times}) = 0$.
Taking the cohomology of the Kummer  exact sequence
$1 \rightarrow \mu_N \rightarrow \C^{\times} \rightarrow \C^{\times} \rightarrow 1$, we deduce that there is an isomorphism
$H^2(A_5,\mu_N) = H^2(A_5,\C^{\times})[N]$, and hence that both groups have order two
if $N$ is even. It follows that $\HN$ is either the unique non-split central
extension or $\HN \simeq A_5 \oplus \Delta$. Since $A_5$ is not a subgroup
of $\GL_2(\Fbar_5)$, it follows that $\HN$ is non-split. Any
 projective morphism $A_5 \rightarrow \PGL_2(\C)$  admits 
 a non-trivial central
extension in $\GL_2(\C)$ of any any even degree --- by uniqueness we
identify these covers with $\HN$ which therefore admits a faithful
representation $\varrho$. From this description, one may compute that the composite:
$$\Delta \rightarrow \HN \hookrightarrow \GL_2(\C)
\rightarrow \C^{\times}$$
surjects onto the image of $\det(\HN)$, 
and thus characters of $\HN$ are determined by their restriction to $\Delta$. 
\end{proof}

Note that $\HN$ (respectively, $A_5$) admits an outer automorphism which is conjugation
by an element of $\GN \setminus \HN$ (respectively, $S_5 \setminus A_5$). We denote this
automorphism by $\iota$ ---  this is not a particularly egregious  abuse of notation since $\iota$
is compatible with the natural projection $\HN \rightarrow A_5$.
We are interested in representations of $\GN$, of $\HN$, and of the group $S_5$.
 The character table of $S_5$ is given as follows:
\begin{center}
\begin{tabular}{|c|c|c|c|c|c|c|c|}
\hline
$\langle g \rangle$ & $1$ & $2A$ & $2B$ & $3A$ & $4A$ & $5A$ & $6A$ \\
\hline
$|\langle g \rangle|$ & $1$ &$10$ & $15$ & $20$ & $30$ & $24$ & $20$ \\
\hline
  $|g|$& $1$ & $2$  & $2$  & $3$ & $4$ & $5$ & $6$ \\
  \hline
  $1$ & $1$ & $1$ & $1$ & $1$ & $1$ & $1$ & $1$ \\
   $\eta$ & $1$ & $-1$ & $1$ & $1$ & $-1$ & $1$ & $-1$ \\
    $\rho_4$ & $4$ & $2$ & $0$ & $1$ & $0$ & $-1$ & $-1$ \\
     $\rho_4 \otimes \eta$ & $4$ & $-2$ & $0$ & $1$ & $0$ & $-1$ & $1$ \\
      $\rho_5$ & $5$ & $1$ & $1$ & $-1$ & $-1$ & $0$ & $1$ \\
       $\rho_5 \otimes \eta$ & $5$ & $-1$ & $1$ & $-1$ & $1$ & $0$ & $-1$ \\
        $\rho_6$ & $6$ & $0$ & $-2$ & $0$ & $0$ & $1$ & $0$ \\
\hline
\end{tabular}
\end{center}

 \begin{df}
If $X$ is a representation of $\HN$ and $Y$ is a representation of $S_5$, say that $X \rightsquigarrow Y$ if the action of $\HN$ on $X$ factors though the quotient $A_5$,
and the action of $A_5$ extends to an action of $S_5$ which is isomorphic to $Y$.
On the other hand, if $X$ is a representation of $\HN$ and $Z$ is a representation of $\GN$, say
that $X \rightsquigarrow Z$ if the action of $\HN$ on $X$ extends to an action of $\GN$ which is isomorphic to 
$Z$.
\end{df}

\begin{lemma} We have the following:  \label{lemma:group}
\begin{enumerate}
\item $(\varrho \otimes \iota \varrho) \otimes \det(\varrho)^{-1} \rightsquigarrow \rho_4$,
\item $\Sym^4(\varrho)  \otimes \det(\varrho)^{-2} \simeq 
\Sym^4(\iota \varrho)  \otimes \det(\varrho)^{-2} \rightsquigarrow \rho_5$,
\item $\wedge^2((\varrho \otimes \iota \varrho) \otimes \det(\varrho)^{-1}) = 
(\Sym^2(\varrho) \otimes \det(\varrho)^{-1}) \oplus 
(\Sym^2(\iota \varrho) \otimes \det(\varrho)^{-1}) \rightsquigarrow \rho_6$.
\end{enumerate}
\end{lemma}

\begin{proof} The image of $\Delta$ under $\varrho$ lies in the scalar matrices, by Schur's lemma.
On the other hand, the involution $\iota$ fixes the centre $\Delta$ of $\HN$,
and hence the restriction of $\iota \varrho$ to $\Delta$ is the same
as the restriction of $\varrho$ to $\Delta$. The product of these
restrictions  it thus identified with $\det(\varrho)$,  and thus the action of the centre $\Delta$ on  $(\varrho \otimes \iota \varrho) \otimes \det(\varrho)^{-1}$
is trivial, and  the action of $\HN$ factors through $S_5$.
On the other hand, the representation is preserved by the automorphism $\iota$, and hence
it lifts to a representation of $S_5$, from which  $(1)$ follows easily.

 The action of the centre on  $\Sym^2(\varrho) \otimes \det(\varrho)^{-1}$ is trivial, and thus it corresponds to a $3$-dimensional representation of $A_5$, which must be one
 of the two irreducible faithful representations of $A_5$.
There is a plethsym:
$$\Sym^2(\Sym^2(\varrho) \otimes \det(\varrho)^{-1}) =  \Sym^4(\varrho) \otimes \det(\varrho)^{-2}
\oplus 1$$
Viewing the left hand side as $\Sym^2$ of an irreducible representation of $A_5$,
we may identify the non-trivial factor on the right hand side as the $5$-dimensional
irreducible representation of $A_5$, which (from the character tables
of $A_5$ and $S_5$) lifts to a representation of $S_5$.
We note, moreover, that this identification can equally be applied to
$\iota \varrho$. This establishes $(2)$.

The claim $(3)$ follows directly from $(1)$, as $\wedge^2 \rho_4 = \rho_6$. 
 \end{proof}

On the other hand, we have the following:

\begin{lemma} There exists an irreducible four dimensional representation
$\xi$ of $\GN$ such that
$$\Sym^3(\varrho) \rightsquigarrow \GN.$$
\end{lemma}

\begin{proof} Let us first consider the projective representation:
$$\PP \Sym^3(\varrho): \HN \rightarrow \PGL_4(\C).$$
The image depends only on the projective image of $\varrho$, and thus
it factors through $A_5$. Let us admit, for the moment, that this projective
representation extends to a projective representation of $S_5$. Then we obtain
an identification of projective representations of $\HN$:
$$\PP \Sym^3(\varrho) \simeq \PP \Sym^3(\iota \varrho).$$
Any two irreducible representations
with isomorphic projective representations are twists of each other.
Since any one dimensional character of $\HN$ is determined by its restriction to
$\Delta$, and since $\Sym^3(\varrho)$ and $\Sym^3 (\iota \varrho)$ are the same
restricted to $\Delta$, it follows that there must be an isomorphism:
$$\Sym^3(\varrho) \simeq \Sym^3(\iota \varrho).$$
It suffices, therefore, to establish that the fact about projective representations
mentioned above. This is a fact which can be determined, for example, by
an explicit computation with the
Darstellungsgruppe of $A_5$, the binary icosahedral group.
Note that although the projective representation of $\xi$ factors through $S_5$,
it is not equivalent to either of the projective representations obtained from the
\emph{linear} representations of $S_5$.
\end{proof}

We now  describe, to some extent, the character $\xi$ of $\GN$.
We start by describing the projective representation associated to $\xi$.
That is, for an element $\sigma \in S_5$, we give a matrix in $\GL_2(\C)$ lifting
$\sigma$. In fact, we only do this
for the \emph{odd} permutations, since this is the only information we shall require. We use the same labeling of elements
as in the character table of $S_5$ above. Here $\ww^{12}= 1$ is a
primitive $12$th root of unity:

\begin{center}
\begin{tabular}{|c|c|c|c|}
\hline
$\langle g \rangle$ & $2A$ & $4A$ &  $6A$ \\
\hline
$|\langle g \rangle|$   & $10$ & $30$ & $20$ \\
\hline
  $|g|$& $2$  & $4$ &  $6$ \\
  \hline   $\eta$ & $-1$ & $-1$ &  $-1$ \\
  \hline $\PP \xi$ &   
  $\left( \begin{matrix} 1 & & & \\ & 1 & & \\ & & -1 & \\ & & & -1 \end{matrix}\right)$ &
   $\left( \begin{matrix} 1 & & & \\ & -1 & & \\ & & i & \\ & & & -i \end{matrix}\right)$
       & 
      $\left( \begin{matrix} \ww^3 & & & \\ & \ww& & \\ & &  \ww^{-1} & \\ & & & \ww^{-3} \end{matrix}\right)$
     \\
     \hline
\end{tabular}
\end{center}

There is an isomorphism $\PP \xi^{\vee} \simeq \PP \xi$, and 
hence an isomorphism $\xi^{\vee} \simeq \xi \psi$ for some character $\psi$.
We have, moreover, that $\xi^{c} \simeq \xi^{\vee}$.

\begin{lemma} The representation $\xi$ preserves
a non-degenerate generalized symplectic pairing: 
$$\xi \times \xi \rightarrow \psi^{-1}.$$
There is an isomorphism
$$(\wedge^2 \chi) \otimes \psi \simeq 1 \oplus \rho_5.$$

\end{lemma}

\begin{proof} Since $\xi^{\vee} \simeq \xi \psi$, 
there is certainly a non-degenerate pairing $\xi \times \xi \rightarrow \psi^{-1}$,
it remains to determine whether this pairing is (generalized) symplectic or orthogonal.
Yet the restriction of $\xi$ to $\HN$ is $\Sym^3(\varrho)$, which is irreducible and symplectic, and thus $\xi$
is symplectic. It follows that $\wedge^2 \xi$ decomposes as $\psi$ plus some $5$-dimensional
representation of $\GN$. Restricting to $\HN$, we have the plethysm
$$(\wedge^2 \Sym^3 \varrho) \otimes \det(\varrho)^{-3} =  1 \oplus (\Sym^4 \varrho) \otimes  \det(\varrho)^{-2}.$$
Since the latter representation extends to $\rho_5$, this shows that
$(\wedge^2 \xi) \otimes \psi$ is either $1 \oplus \rho_5$ or $1 \oplus \rho_5 \otimes \eta$. To
distinguish between these two representations, let us compute $\wedge^2 \xi$ on some
conjugacy class $\sigma$ in $\GN$ which maps to $6A$ in $S_5$. We must have
$$\xi(\sigma) = \left( \begin{matrix} \ww^3 \zeta & & & \\ & \ww \zeta & & \\ & &  \ww^{-1} \zeta & \\ & & & \ww^{-3} \zeta \end{matrix}\right)$$
for some root of unity $\zeta$.
It follows that $\wedge^2 \xi(\sigma)$ has eigenvalues:
$$\{\zeta^2,\zeta^2, \ww^2 \zeta^2, \ww^{-2} \zeta^2, \ww^{4} \zeta^2, \ww^{-4} \zeta^2\}$$
  Note that since $\xi$ is (generalized)
symplectic with similude character $\psi^{-1}$, the eigenvalues of $\xi(\sigma)$ are of the form
$\{\alpha, \beta, (\alpha \psi(\sigma))^{-1}, (\beta \psi(\sigma))^{-1} \}$. It follows
that $\psi^{-1}(\sigma)$ occurs to multiplicity at least two in the eigenvalues of $\wedge^2 \xi(\sigma)$, and thus
 $\psi^{-1}(\sigma) = \zeta^2$. In particular, we deduce that
$$\Tr((\wedge^2 \xi) \otimes \psi)(\sigma))
= 1 + 1 + \ww^2 + \ww^{-2} + \ww^{4} + \ww^{-4} = 2.$$
Yet, if $\sigma$ has conjugacy class $6A$ in $S_5$, then
$\Tr((1 \oplus \rho_5)(\sigma)) = 2$ 
whereas $\Tr((1 \oplus \rho_5 \otimes \eta)(\sigma)) = 0$.
\end{proof}

\begin{lemma} \label{lemma:char}
Let $\sigma \in \GN$ be a lift of $2A$, and suppose that
$$\xi(\sigma) = \left( \begin{matrix} \zeta & & & \\ & \zeta & & \\ & & -\zeta & \\ & & & -\zeta \end{matrix}\right).$$
Then $\psi(\sigma) = -\zeta^{-2}$.
\end{lemma}

\begin{proof} Suppose instead  that $\psi(\sigma) =  \zeta^{-2}$. Then
$(\wedge^2 \xi)(\sigma)$ has eigenvalues:
$$\{\zeta^2,\zeta^2,-\zeta^2,-\zeta^2,-\zeta^2,-\zeta^2\}$$
If $\psi(\sigma) = \zeta^{-2}$, then the trace of $(\wedge^2 \otimes \psi^{-1})(\sigma)$ is
$-2$, whereas if $\psi(\sigma) = -\zeta^{_2}$ then this trace is $2$. Yet if $\sigma$ is of type $2A$, then
$\Tr((1 \oplus \rho_5)(\sigma)) = 2$.
\end{proof}

\section{Irreducible representations of dimension $n = 4$ and $n = 6$}
Let $\K/\Q$ be an $S_5$-extension satisfying the conditions of Theorem~\ref{theorem:main}. Let $E/\Q$ denote the quadratic subfield of $\K$.
By assumption, $E$ is real, and $5$ and $3$ both split in $E$. 

\begin{lemma} There exists an Galois extension $\L/\Q$ with $\Gal(\L/\Q) \simeq \GN$
for some $N$.
\end{lemma}

\begin{proof}
There is an exceptional isomorphism $\PGL_2(\F_5) \simeq S_5$, giving rise to
a representation
$$\Gal(\Qbar/\Q) \rightarrow \PGL_2(\Fbar_5)$$
which factors through $\Gal(\K/\Q)$. The Galois
cohomology group
$H^2(G_{\Q},\Fbar^{\times}_5)$ is trivial by a theorem of Tate (see~\cite{Serre}). It follows that
the projective representation lifts to a linear representation over $\Fbar_5$ with finite image.
By a classification of subgroups of
$\GL_2(\Fbar_p)$, the image must land in $\GL_2(\F_5) \Fbar^{\times}_5$, proving the
lemma. 
\end{proof}

It follows that there exists a corresponding complex representation:
$$\varrho: G_{E} \rightarrow \Gal(\L/E) \rightarrow \GL_2(\C)$$
with projective image $A_5$. By assumption, $\varrho$ is odd at
both real places of $E$.

\begin{lemma} \label{lemma:sas} The representation $\varrho$ is automorphic for $\GL(2)/E$.
\end{lemma}

\begin{proof} This follows immediately from Theorem~1 of~\cite{Sasaki2}, given the assumed conditions on 
$\varrho(\Frob_5)$.
\end{proof}

Assume that $\rho$ is irreducible and of dimension $n = 4$ or $n = 6$.
Let $\pi_E$ denote the automorphic form for $\GL(2)/E$ associated to $\varrho$ by Lemma~\ref{lemma:sas}.
There exists a conjugate representation $\iota(\varrho)$ given by applying the outer involution to $\varrho$, which
corresponds (by another application of Lemma~\ref{lemma:sas}) to an automorphic form which we call 
$\iota(\pi_E)$.

Suppose that $n  = 4$.
The Asai transfer $\As(\pi_E)$ is automorphic and cuspidal for $\GL(4)/\Q$, as follows
by  Theorem~D of~\cite{Asai} and Theorem~B of~\cite{RP}. Yet, by Lemma~\ref{lemma:group} 
part~$(1)$, we
deduce that $\As(\pi_E)$ corresponds (by Lemma~\ref{lemma:group})  exactly to
a twist of $\rho_4$, 
from which the main result follows.

Suppose that $n = 6$. Then $\rho = \rho_6 = \wedge^2 \rho_4$. We have already shown that $\rho_4$ is automorphic and corresponds to some
cuspidal $\pi$ for $\GL(4)/\Q$. By a result of Kim~\cite{K}, the representation
 $\wedge^2 \pi$ is automorphic for $\GL(6)/\Q$, and by~\cite{asgari},
it is a simple matter to check that it is cuspidal. Hence $\rho_6$ corresponds (by Lemma~\ref{lemma:group})  to
$\wedge^2 \pi$, from which the main result follows. Alternatively, we may consider the automorphic form $\Sym^2 \pi_E$.
The corresponding Galois representation $\Sym^2(\varrho)  \otimes \det(\varrho)^{-1}$ has image $A_5$, but is not isomorphic to its
Galois conjugate (as $S_5$ has no irreducible representations of dimension three). Thus one may form the base change of $\Sym^2(\pi_E)$ to $\GL(6)/\Q$, which is cuspidal
(by cyclic base change) and corresponds to the Galois representation $\rho_6$. 

\section{Irreducible representations of dimension $n = 5$}

Proving Langlands' automorphy conjecture for $\rho = \rho_5$ or $\rho_5 \otimes \eta$ is somewhat harder, 
for reasons which we now explain.
The symmetric fourth power $\Sym^4(\pi_E)$ of $\pi_E$ 
is automorphic and cuspidal by~\cite{K} and~\cite{KS} respectively.
On the other hand, $\Sym^4(\varrho) \otimes \det(\varrho)^{-2}$
 is invariant under the involution of $\Gal(E/\Q)$, and thus, by
multiplicity one~\cite{Ja} and cyclic base change (Theorem~4.2 (p.202)  of~\cite{AC}),
the corresponding twist  
$\Sym^4(\pi_E) \times \det(\pi_E)^{-2}$
 arises from an automorphic form $\varpi$ for $\GL(5)/\Q$.
The Galois representation $\rho$ restricted to $G_{E}$ corresponds to this automorphic form. We would like to show that
$\rho$ (or its quadratic twist)
corresponds to $\varpi$. There is a well known problem, however, that since this descended form is not known
a priori to admit a Galois representation,  all we can deduce is that the collection of Satake parameters for
$\varpi$ and $\varpi \otimes \eta$ are the same as the collection
of Frobenius eigenvalues of $\rho$ and $\rho \otimes \eta$. This implies that
$L(\rho,s) L(\rho \times \eta,s)$ is holomorphic, but gives no information about
$L(\rho,s)$. This problem is the main obstruction
to proving the  Artin conjecture for solvable representations.

%

\medskip

The usual game in these situations is to play off several representations against each other and use functorialities known in small degrees. It turns out --- in this instance ---
that it is profitable to work instead with the representation $\xi$ of the group
$\GN$ in dimension four.
Let $\varrho$ denote one of the two dimensional representations
of $\HN$. Since $\Sym^3(\varrho)$ is equal to its conjugate under $\Gal(E/\Q)$ (by Lemma~\ref{lemma:group}),
we deduce by multiplicity one that the same is true of
$\Sym^3(\pi_E)$. It follows by cyclic base change that there
exists an automorphic form $\Pi$ for
$\GL(4)/\Q$ such that $\Pi_E \simeq \Sym^3(\pi_E)$. Conjecturally, $\Pi$ is associated
to the Galois representation:
$$\xi: \widetilde{\Sym^3(\varrho)}: \Gal(\L/\Q) \simeq \GN \hookrightarrow GL_4(\C).$$

\medskip

 Let $S_x$ denote the $4$-tuple of the Satake parameters of $\Pi_x$.
 By abuse of notation, we write $\xi(x)$ and $\eta(x)$ for
 $\xi(\Frob_x)$ and $\eta(\Frob_x)$ respectively.
 
 \begin{lemma} If $\eta(x) = +1$, then $S_x$ consists
 of the eigenvalues of $\xi(x)$.
 If $\eta(x) = -1$, then the union $S_x \cup -S_x$ consists
 of the eigenvalues of $\xi(x)$ together with the eigenvalues
 of $-\xi(x)$.
 \end{lemma}
 
 \begin{proof} This follows from the local compatibility between
 $\Pi_E$ and $r|G_{E}$. 
 \end{proof}
 
Since $\Pi^{\vee}_E \simeq \Pi_E \otimes \psi_E$ (by multiplicity one) we
deduce (again by multiplicity one)
 that $\Pi^{\vee} \simeq \Pi \otimes \nu^{-1}$ for some 
character $\nu$ such that $\nu_E \psi_E$ is trivial. It follows that either
 $\nu = \psi^{-1}$ or $\nu = \eta \psi^{-1}$. Since $\xi^{\vee} \simeq \xi \psi$,
 it should be the case that $\nu = \psi^{-1}$, although there does
 not seem to be any apparent way to prove this a priori. 
 We shall, however, prove that $\nu = \psi^{-1}$ 
 in Lemma~\ref{lemma:sim} below.

\medskip

\begin{lemma} \label{lemma:satake} $\Pi$ is of symplectic type, and
$\nu$ is the corresponding similitude character.
If $x$ is an unramified prime, then the Satake parameters of
$\Pi_x$ are of the form:
$$\{\alpha,\beta,  \nu(x)/\alpha, \nu(x)/\beta\}.$$ 
\end{lemma}

\begin{proof} 
It suffices to show that $\wedge^2 \Pi$ is non-cuspidal, and then 
we can deduce the
result from Theorem~1.1(ii) of~\cite{asgari}. Since $\Pi_E$ is symplectic, we know
that $\wedge^2 \Pi_E$ is non-cuspidal. Assume that $\wedge^2 \Pi$ is
cuspidal. Then the base change of $\wedge^2 \Pi$ to $E$ is $\wedge^2 \Pi_E$.
In other words, $\wedge^2 \Pi$ becomes non-cuspidal after base change.
By Theorem~4.2 (p.202)  of~\cite{AC}, it follows that
$\wedge^2 \Pi$ is the automorphic induction of a cuspidal form $\mu$
from $\GL(3)/E$. We deduce that  $\wedge^2 \Pi_E = \mu \boxplus \mu$, which
is incompatible with the fact that the Satake parameters of $\wedge^2 \Pi_E$
do not all have multiplicity two. Thus $\wedge^2 \Pi$ is not cuspidal,
and from the classification~\cite{asgari}, we deduce that $\Pi$ is symplectic (all other possibilities
contradict the known structure of $\Pi_E$).  If $\Pi$ has similitude character
$\widetilde{\nu}$, then $\Pi \simeq \Pi^{\vee} \otimes \widetilde{\nu}$. It follows that
$\Pi \simeq  \Pi^{\vee} \otimes \widetilde{\nu} \simeq \Pi \otimes  \widetilde{\nu} \nu^{-1}$. It is easy to check that $\Pi$ is not induced from
a quadratic subfield, and hence we must have $\widetilde{\nu} \simeq \nu$.
 \end{proof}

The second wedge $\wedge^2 \Pi$  decomposes as an isobaric sum
of the similitude character $\nu$ together with another isobaric representation,
 and thus we may write
character $\nu^{-1}$, and thus we may write
$$(\wedge^2 \Pi) \otimes \nu^{-1}
 =  \varpi \boxtimes 1.$$
 By considering Galois representations, we see that
$\varpi_E$ corresponds to $\rho|_{E}$, and $\varpi$ conjecturally corresponds to $\rho_5$.
It also follows from this that $\varpi$ is cuspidal for $\GL(5)/\Q$.


\medskip

Let $T_x$ denote the $5$-tuple of Satake parameters of $\varpi$. 
If $S_x = \{\alpha,\beta,\nu(x)/\alpha, \nu(x)/\beta\}$, 
then  the Satake parameters $T_x$ of $\varpi$ are 
of the form
$$\left\{1,\alpha \beta/\nu(x),  \nu(x)/\alpha \beta, \frac{\alpha}{\beta}, \frac{\beta}{\alpha} \right\}$$

\begin{df} For the automorphic representations $\Pi$ and $\varpi$, 
let 
$$\chi(\Pi,x) = \sum_{S_x} \alpha, \qquad \chi(\varpi,x) = \sum_{T_x} \alpha$$
denote the sum of the Satake parameters. For the Galois representations $\xi$ and
$\rho$, let
$\chi(\xi,x)$ and $\chi(\rho,x)$ denote the sum of the eigenvalues of 
$\xi(x)$ and $\rho(x)$ respectively.
\end{df}

Conjecturally, we have $\chi(\Pi,x) = \chi(\xi,x)$ and $\chi(\varpi,x) = \chi(\rho,x)$.
Note that $\chi(\rho,x)$ and $\chi(\varpi,x)$ are real valued; the former
because $\rho$ is a real representation, and the latter because $T_x$ consists
of $1$ together with two pairs $\{\zeta,\zeta^{-1}\}$ for a root of unity $\zeta$.
We shall consider the following quantity:

\begin{df}
For a prime $x$, let
$$\xxi(x) = \|\chi(\Pi,x)\|^2 - \|\chi(\xi,x)\|^2 - (\chi(\varpi,x)^2 - \chi(\rho,x)^2).$$
\end{df}

If $\eta(x) = 1$ then $S_x$ consists of the eigenvalues of $\xi(x)$,
and $T_x$ consists of the eigenvalues of $\rho(x)$, and hence $\xxi(x) = 0$.

 We now consider the possible values of $\xxi(x)$ 
as well as the structure of $T_x$ for $x$ with $\eta(x) = -1$.
Note that we do not know at this point whether $\nu \simeq \psi^{-1}$
or $\nu \simeq \psi^{-1} \eta$, so we shall have to take both possibilities
into account.

\begin{enumerate}
\item Suppose that the conjugacy class of $\Frob_x$ is $2A$ in $S_5$.
Then from the (projective) character table of $\chi$, we deduce that
$$S_x \cup -S_x =
\{\zeta,\zeta,-\zeta,-\zeta\} \cup \{\zeta,\zeta,-\zeta,-\zeta\}.$$
By Lemma~\ref{lemma:char}, we know that $\psi(x) = -\zeta^{-2}$.
Hence either $\nu(x) = - \zeta^2$ if $\nu = \psi$, or $\nu(x) = \zeta^2$ if
$\nu = \psi \eta$.
Using the known shape of $S_x$, we
 deduce that one of the following possibilities holds
 (we only concern ourselves with identifying $S_x$ up to sign.)
 \begin{center}
 \begin{tabular}{|c|c|c|c|c|c|c|}
 \hline
  $\nu^{-1}$ &  $\nu(x)$ & $\pm S_x$ & $\|\chi(\Pi,x)\|^2$ & $T_x$ & $\chi(\varpi,x)^2$ & $\xxi(x)$ \\
  \hline
  $\psi$ & $-\zeta^2$ &   $\{\zeta,\zeta,-\zeta,-\zeta\}$ & $0$ & $\{1,1,1,-1,-1\}$ &  $1$ &  $0$ \\
   $\psi \eta$ & $\zeta^2$ &    $\{\zeta,\zeta,-\zeta,-\zeta\}$ & $0$ & $\{1,-1,-1,-1,-1\}$ &  $9$ & $-8$ \\
   $\psi \eta$ & $\zeta^2$ &   $\{\zeta,\zeta,\zeta,\zeta\}$ & $16$ &  $\{1,1,1,1,1\}$ & $25$ & $-8$  \\
  \hline
  \end{tabular}
  \end{center}
  Both here and in the two tables below, the first line of each table represents the
  (conjectural) reality.
  \item Suppose that the conjugacy class of $\Frob_x$ is $4A$ in $S_5$.
  Then from the (projective) character table of $\chi$, we deduce that
 $$S_{x} \cup - S_x = \{\zeta,i \zeta,-\zeta,-i \zeta \} \cup \{\zeta,i \zeta,-\zeta,-i \zeta \}.$$
 By computing multiplicities in $\wedge^2 S_x$ we deduce that $\psi(x) = i \zeta^2$ or $- i \zeta^2$
 but we cannot pin down $\psi(x)$ exactly --- indeed, the group
 $\GN$ can (and does) contain distinct conjugacy classes with these same eigenvalues
 and with values of $\psi$ that differ (up to sign). It follows that the possibilities below
 are the same regardless whether $\nu = \psi^{-1}$ or $\nu = \psi^{-1} \eta$.
  \begin{center}
 \begin{tabular}{|c|c|c|c|c|c|}
 \hline
  $\nu(x)$ & $\pm S_x$ & $\|\chi(\Pi,x)\|^2$ & $T_x$ & $\chi(\varpi,x)^2$ & $\xxi(x)$ \\
  \hline
  $i \zeta^2$ &  $\{\zeta,i \zeta,-\zeta,-i \zeta \}$ & $0$ &$\{1,i,-i,-1,-1\}$ & $1$ &  $0$\\
  $i \zeta^2$ & $\{\zeta,i \zeta,\zeta,i \zeta \}$ & $8$ &   $\{1,i,-i,1,1\}$ & $9$ &  $0$ \\
  $-i \zeta^2$ & $\{\zeta,i \zeta,-\zeta,-i \zeta \}$ & $0$ &  $\{1,i,-i,-1,-1\}$ & $1$ & $0$ \\
   $-i \zeta^2$ &  $\{\zeta,-i \zeta,\zeta,-i \zeta \}$ & $8$ &  $\{1,i,-i,1,1\}$  & $9$  &  $0$ \\
  \hline
  \end{tabular}
  \end{center}
  \item Suppose that the conjugacy class of $\Frob_x$ is $6A$ in $S_5$. Then, as above, we may
write
 $$S_{x} \cup - S_x = \{ \ww^{3}\zeta, \ww \zeta, \ww^{-1} \zeta, \ww^{-3} \zeta\} \cup 
  \{- \ww^{3} \zeta,- \ww \zeta,- \ww^{-1} \zeta, - \ww^{-3} \zeta\},$$
  where $\ww^{12} = 1$. A key point in the computation below is that $\ww^6 = - 1$, and so
  $\ww^{-3} = - \ww^3$. Here $\psi(x) = \zeta^{-2}$, and so $\nu(x) = \zeta^2$ if
  $\nu = \psi^{-1}$ and $-\zeta^2$ otherwise.
We deduce that the following possibilities may occur:
 \begin{center}
 \begin{tabular}{|c|c|c|c|c|c|c|}
 \hline
  $\nu^{-1}$ &  $\nu(x)$ & $\pm S_x$ & $\|\chi(\Pi,x)\|^2$ & $T_x$ & $\chi(\varpi,x)^2$ & $\xxi(x)$ \\
  \hline
  $\psi$ &  $\zeta^2$ & $ \{ \ww^{3}\zeta, \ww \zeta, \ww^{-1} \zeta, \ww^{-3} \zeta\}$ & $3$ & 
  $\{1,\ww^2,\ww^{-2},\ww^4,\ww^{-4}\}$
  &  $1$ &  $0$ \\
   $\psi \eta$ &  $-\zeta^2$ &  $\{\ww^3 \zeta, \ww \zeta,  -\ww^{-1} \zeta, -\ww^{-3} \zeta\}$ & $9$ & 
   $\{1,\ww^2,\ww^{-2},-\ww^4,-\ww^{-4}\}$
 &  $9$ & $2$ \\
   $\psi \eta$ & $-\zeta^2$ &  $\{\ww^3 \zeta,  - \ww \zeta, \ww^{-1} \zeta, -\ww^{-3} \zeta\}$ & $1$ &
    $\{1,-\ww^2,-\ww^{-2},\ww^4,\ww^{-4}\}$  & $1$ & $2$  \\
  \hline
  \end{tabular}
  \end{center}
\end{enumerate}

\begin{lemma}  \label{lemma:JS}
The function
$$\frac{L(\Pi \times \overline{\Pi},s)}{L(\varpi \times \varpi,s)} \cdot \frac{L(\rho_5 \times \rho_5,s)}
{L(\xi \times \overline{\xi},s)}$$
is meromorphic for $\Re(s) > 0$, and is holomorphic in some neighbourhood
of $s = 1$.
\end{lemma}

\begin{proof} By
By a theorem of Jacquet and Shalika (\cite{Jb}), the Rankin--Selberg
$L$-functions are meromorphic for $\Re(s) > 0$ with a simple pole each
at $s = 1$. The same is true of the Artin $L$-functions by Brauer's theorem.
\end{proof}

\begin{lemma} There is an equality $\nu = \psi^{-1}$. \label{lemma:sim}
\end{lemma}

\begin{proof} Assume otherwise. From the Euler product, we find that, as $s \rightarrow 1^{+}$,
$$\log \left| \frac{L(\Pi \times \overline{\Pi},s)}{L(\varpi \times \varpi,s)} \cdot \frac{L(\rho_5 \times \rho_5,s)}
{L(\xi \times \overline{\xi},s)}  \right| = \sum \frac{\xxi(p)}{p^s} + O(1).$$
Then, from the tables above, we compute
that $\xxi(x) = 0$ unless  the projective image of $\xi(\Frob_x)$ is of type $2A$, in which case it is $-8$,
or $6A$, in which case it is $+2$.
By The Cebotarev density theorem, the class $2A$ has density
$1/12 = 10/120$ each respectively. Similarly, the class $6A$ has
 Dirichlet density $1/6$. Thus the RHS is asympotic to
$$\log | (s-1)| \left( \frac{-8}{12} + \frac{2}{6} \right) + O(1) =  - \frac{1}{3} \log|(s-1)| + O(1).$$
 This contradicts Lemma~\ref{lemma:JS}.
 \end{proof}

\begin{theorem} Let $B$ denote the finite set of places for
which $\varpi_p$ is not unramified.  For all primes outside a set $\Omega \cup B$
of Dirichlet density zero, $T_x$ consists of the eigenvalues of $\rho(\Frob_x)$.
If $x$ lies in $\Omega$, then $\Frob_x$ in $S_5$
has conjugacy class $4A$, 
and $T_x = \{1,i,-i,1,1\}$ rather than $\{1,i,-i,-1,-1\}$.
\end{theorem}

\begin{proof} Since $\nu = \psi^{-1}$, the result follows automatically for all conjugacy classes
by our computation above except for the density claim concerning
$\Omega$.
Assume otherwise. Denote the set of such primes in these conjugacy classes with
$\|\chi(\Pi,x)\|^2 = 8$ by $\Omega$.
Then we compute that
$$- \log \left| \frac{L(\Pi \times \overline{\Pi},s)}{L(\xi \times \overline{\xi},s)} \right| \sim  
\sum_{\Omega} \frac{8}{p^s} + O(1).$$
Once more by Jacquet--Shalika, we deduce that the LHS is bounded,
and hence the RHS also has
order $o(|\log(s-1)|^{-1})$, 
and hence that $\Omega$ has Dirichlet density zero.
\end{proof}

We now summarize what we have shown so far:
Namely, in comparing the Artin representation $\xi$ with the automorphic form $\Pi$, we have
that for a set of $x$ outside the set $\Omega$ of density zero,
the Satake parameters $S_x$  of $\Pi_x$ agree with
$\xi(\Frob_x)$ up to sign. In particular, outside the same set,
the Satake parameters $T_x$ of  $\varpi_x$ agree with the eigenvalues of
$\rho_5(\Frob_x)$.
This completes the proof of Theorem~\ref{theorem:main}.

Unfortunately, we do not see an unconditional argument at this point for
establishing an equivalence at all places. On the other hand, we
know (by construction) precisely the Satake parameters of $\varpi$
at the ``troublesome'' primes $p \in \Omega$. The special form of these parameters
will allow us to prove Theorem~\ref{theorem:two}.

\begin{lemma} \label{lemma:aa} Let $\mu(s):=\mu_B(s) \mu_{\Omega}(s)$, where $\mu_{\Omega}(s)$ denotes the function 
$$\mu_{\Omega}(s):= \prod_{\Omega} \left( \frac{ 1 + \frac{1}{p^s}}{1 - \frac{1}{p^s}} \right)^2.$$
and 
$\displaystyle{\mu_{B}(s) =  \frac{L_{\infty}(\varpi,s)}{L_{\infty}(\rho_5,s)} \prod_{B} \frac{L_p(\varpi,p^{-s})}{L_p(\rho_5,p^{-s})}}$, 
where the product is over the finite set of primes of bad reduction of $\varpi$ and $\rho_5$.
Then the following holds:
\begin{enumerate}
\item There is an equality $L(\varpi,s) = L(\rho_5,s) \mu(s)$,
\item $\mu(s)$ extends to a meromorphic function on the complex plane.
\item $\mu_{B}(s)$ is holomorphic and non-vanishing on the interval $s \in (0,1)$.
\item Either $\Omega$ is finite, or $\mu(s)$ has a pole
on the real axis with $s \in (0,1)$.
\end{enumerate}
\end{lemma}

\begin{proof} We have that $L(\varpi,s) = L_{\infty}(\varpi,s) \prod L_p(\varpi,p^{-s})^{-1}$ and $L(\rho_5,s) = L_{\infty}(\rho_5,s) \prod L_p(\rho \otimes \eta,p^{-s})$,
where the local factors agree for $p \notin \{\infty\} \cup B \cup \Omega$. For the exceptional $p \in \Omega$, the corresponding polynomials are:
$$L_p(\varpi,X) = (1 + X^2)(1 - X)^3, \qquad L_p(\rho,X) = (1 + X^2)(1-X)(1+X)^2.$$
 Comparing the two sides leads to the equality $(1)$. 
 We claim that the  Gamma factors of both $L$-functions involve only 
the standard $\Gamma$ factors $\Gamma_{\R}(s)$ and $\Gamma_{\C}(s)$. This is
true for $L_{\infty}(\rho_5,s)$ by Artin, and is true for $L_{\infty}(\varpi,s)$ because
of the identity:
$$L_{\infty}(\varpi,s) L_{\infty}(\varpi \otimes \eta,s) = L_{\infty}(\varpi_{E},s).$$
Since $\Gamma_{\R}(s)$ and $\Gamma_{\C}(s)$ are both holomorphic and without
zeros on $(0,1)$, so is any ratio of products of such functions.
The
 factors $L_p(\rho,p^{-s})$ and $L_p(\varpi,p^{-s})$ for finite bad primes 
 are of the form
$P(p^{-s})$ where $P(T)$ is a polynomial whose roots are roots of unity. 
 In particular, the
only poles and zeros of $\mu_B(s)$ occur for complex $s$ such that $p^{ns} = 1$ for some
$n \in \Z$, which cannot happen if $s \in (0,1)$ is real.  Hence $\mu_{B}(s)$ is holomorphic
and non-vanishing on $(0,1)$.
 It therefore suffices to show that if $\Omega$ is infinite, then
$\mu_{\Omega}(s)$ has a pole for $s \in (0,1)$.
The Taylor series coefficients of
$$\frac{(1+x)^2}{(1-x)^2} = 1 + \sum_{n=1}^{\infty} 4n x^n$$
are all positive, and hence the Dirichlet series for $\mu_{\Omega}(s)$ has positive real terms.
It follows that $\mu_{\Omega}(s)$ is strictly increasing as a real function of $s$ as $s$ decreases along
the real axis (in the range where $\mu_{\Omega}(s)$ is convergent).  A standard argument then implies that
$\mu_{\Omega}(s)$ (which is meromorphic) must have a pole in $(0,1)$.
\end{proof}

We now upgrade this lemma to deduce that if $\Omega$ is finite, then 
$L(\varpi,s)$ is equal to $L(\rho_5,s)$ on the nose.

\begin{lemma} If $\Omega$ is finite, then $L(\varpi,s) = L(\rho_5,s)$. \label{lemma:bb}
\end{lemma}

\begin{proof} It suffices to show that the
$L$-factors agree at any place. By assumption, there
exists a finite set $S$ of places at which they differ. 
The proof is similar to Proposition~4.1 of~\cite{Asai}, which we follow closely,
although it is easier, because we have more explicit information about the $\Gamma$-factors
at infinity.
 Indeed, as in \emph{ibid}., we may find a ramified  character $\rec$  such that
 the $L$-factors of $\varpi \times \rec$ and  $\rho_5 \otimes \rec$ are trivial at finite
 places dividing $S$. If we furthermore assume that $\rec$ is real, then the corresponding
 $L$-factors
 at infinity do not change. 
  Then, from the functional equations, we deduce that
$$L_{\infty}(s,\varpi) L_{\infty}(\rho_5,1-s) = L_{\infty}(1-s,\varpi) L_{\infty}(\rho_5,s).$$
Each of these $L$-factors $L_{\infty}(s)$ is a product of terms of the form $\Gamma_{\R}(s)$ and
$\Gamma_{\C}(s)$, and thus is holomorphic and invertible for $\Re(s) > 0$. It follows that
we may identify the polar divisor of $L_{\infty}(\rho_5,s)$ with the polar divisor
of $L_{\infty}(\varpi,s)$, and then deduce they are equal, by the Baby Lemma of~\cite{Asai}.
The equality at the remaining finite places then follows by twisting with a character $\rec$
that is highly
ramified at all but one finite place $v$ in $S$, and split completely at $v$. Comparing functional
equations as in the archimedean case, we deduce an equality of $L$-factors at $v$.
\end{proof}

We now complete the proof of Theorem~\ref{theorem:two}.
Since $L(\varpi,s)$ is holomorphic, it follows that
if $\mu(s)$ has a real pole for $s \in (0,1)$, then $L(\rho_5,s)$ has a zero in the same interval.
On the other hand,
$$\zeta_{H}(s) =  \zeta_{F}(s) L(\eta,H/F,s) = \zeta_{F}(s) L(\eta,E/\Q,s) L(\rho_5,s),$$
and by results of Hecke and Riemann,  $\zeta_{F}(s)$ and $L(\eta,E/\Q,s)$ are holomorphic
in $s \in (0,1)$, and so $\zeta_{H}(s)$ also has a zero in this interval.
If not, then by Lemma~\ref{lemma:aa}, $\Omega$ is finite, 
and by Lemma~\ref{lemma:bb},  we have equalities $L(\varpi,s) = L(\rho_5,s)$ and
$L(\varpi \times \eta,s) = L(\rho_5  \otimes \eta,s)$, implying that the latter functions are automorphic.
This completes the proof of Theorem~\ref{theorem:two}.

%

\section{An Example}

\label{section:comp}

As an example of Theorem~\ref{theorem:two},  with help from a custom computation done for us by Andrew
 Booker,  we prove the Artin conjecture for the Galois
closure of the quintic field of
smallest discriminant.

\begin{theorem} \label{example:booker}
Let $K = \Q(x)/(x^5 - x^3 - x^2 + x + 1)$.
and let $\K$ denote the Galois closure of $K$. Then any complex representation of $\Gal(\K/\Q)$ is automorphic.
\end{theorem}

\begin{proof}
The discriminant of $K$ is $\Delta_K = 1609$,
there is an isomorphism $G:=\Gal(\K/\Q) = S_5$, and $K$ has signature $(1,2)$, so complex conjugation
is conjugate to $(12)(34)$. The prime $5$ is inert in $\OL_K$, however, so one cannot apply Sasaki's
theorem directly. 
Instead, let us consider the corresponding mod-$2$ representation
$$\varrhobar: G_{\Q(\sqrt{1609})} \rightarrow \SL_2(\F_4)$$
that  is unramified at all finite places. By Theorem~2 of~\cite{Sasaki1}, to establish the modularity of the
complex two dimensional representation with projective image $\Gal(\K/\Q)$, it suffices to prove
that $\varrhobar$ is modular. (Note that $2$ is totally split in $E$, and unramified in $\K$, and that
$\varrhobar$ is $2$-distinguished because $\Frob_2$ has order $5$ in $S_5$.)
\begin{lemma}
The representation $\varrho$ is modular of weight $(2,2)$ and level one for $\GL(2)/E$.
\end{lemma}

\begin{proof} Using a {\tt magma} program written by Lassina Dembele, 
one may verify the following facts:
\begin{enumerate}
\item There exists a form $f$ with coefficients in $\F_{4}$ which
is an eigenform for all the Hecke operators $T_{\p}$ with $N(\p)$ odd.
\item The image of $\varrhobar_{f}$ surjects onto $\SL_2(\F_4)$. 
\item If $\sigma(f)$ denotes the conjugate of $f$ by $\Gal(E/\Q)$, and $\Frob(f)$ denotes the image of $f$
under the Frobenius automorphism acting on the coefficient field $\F_4$, then $f \neq \sigma(f)$, but
$$\sigma(f) =\Frob(f).$$
\end{enumerate}
Since the kernel of $\varrhobar_{\Frob(f)}$ 
is the same as the kernel of $\varrhobar_f$,
we deduce that
this kernel defines an $A_5$-extension of $E$ which is Galois over $\Q$. Since
$f \ne \sigma(f)$, it follows that this must be  a non-split extension, and thus define an $S_5$-extension
of $\Q$. Using Fontaine's bounds~\cite{Fontaine} for root discriminants of finite flat group schemes, we deduce that
the discriminant of quintic subfield divides $1609 \cdot 2^9$ (the $2$-adic valuation of the root discriminant must
be strictly less than $1 + 1/(2-1) = 2$). The only quintic subfield with this property
is $K$ (see~\cite{Diaz}), and so 
$\varrhobar_f = \varrhobar$,
which is thus modular.
(Note that $\Frob$ acting on $\SL_2(\F_4)$ is the outer automorphism $\iota$ of $A_5$
discussed previously coming from the inclusion of $A_5$ in $S_5$.)
\end{proof}

As in proof of Theorem~\ref{theorem:two}, it suffices to  prove that $L(\rho_5,s)$ does not vanish for $s \in (0,1)$.
In~\cite{Booker}, Booker found an algorithm that allows one to unconditionally verify the Artin conjecture and the GRH
for Artin $L$-functions of $S_5$-representations (and many other groups) in any bounded range within the critical strip.
Booker has carried out his algorithm in this case, and one finds the following:

\begin{theorem}[Booker] \label{theorem:booker} The lowest lying zero of $L(\rho_5,s)$
occurs at 
$$s = \frac{1}{2} + \gamma i, \qquad \text{where \ } \gamma =   1.624\ldots > 0.$$
$($Moreover, all the zeros of $L(\rho_5,s)$ with $|\Im(s)| < 100$ lie on the critical  line.$)$
\end{theorem}

In particular, $L(\rho_5,s)$ does not vanish for real $s \in (0,1)$, proving Theorem~\ref{example:booker}.
\end{proof}

\bibliographystyle{amsalpha}
\bibliography{S5}

\providecommand{\bysame}{\leavevmode\hbox to3em{\hrulefill}\thinspace}
\providecommand{\MR}{\relax\ifhmode\unskip\space\fi MR }
\providecommand{\MRhref}[2]{%
  \href{http://www.ams.org/mathscinet-getitem?mr=#1}{#2}
}
\providecommand{\href}[2]{#2}
\begin{thebibliography}{BDSBT01}

\bibitem[AC89]{AC}
James Arthur and Laurent Clozel, \emph{Simple algebras, base change, and the
  advanced theory of the trace formula}, Annals of Mathematics Studies, vol.
  120, Princeton University Press, Princeton, NJ, 1989. \MR{MR1007299
  (90m:22041)}

\bibitem[AR11]{asgari}
Mahdi Asgari and A.~Raghuram, \emph{A cuspidality criterion for the exterior
  square transfer of cusp forms on {${\rm GL}(4)$}}, On certain
  {$L$}-functions, Clay Math. Proc., vol.~13, Amer. Math. Soc., Providence, RI,
  2011, pp.~33--53. \MR{2767509}

\bibitem[Arm72]{Arm}
J.~V. Armitage, \emph{Zeta functions with a zero at {$s={1\over 2}$}}, Invent.
  Math. \textbf{15} (1972), 199--205. \MR{0291122 (45 \#216)}

\bibitem[BDSBT01]{BDST}
Kevin Buzzard, Mark Dickinson, Nick Shepherd-Barron, and Richard Taylor,
  \emph{On icosahedral {A}rtin representations}, Duke Math. J. \textbf{109}
  (2001), no.~2, 283--318. \MR{1845181 (2002k:11078)}

\bibitem[Boo06]{Booker}
Andrew~R. Booker, \emph{Artin's conjecture, {T}uring's method, and the
  {R}iemann hypothesis}, Experiment. Math. \textbf{15} (2006), no.~4, 385--407.
  \MR{2293591 (2007k:11084)}

\bibitem[Fon85]{Fontaine}
Jean-Marc Fontaine, \emph{Il n'y a pas de vari\'et\'e ab\'elienne sur {${\bf
  Z}$}}, Invent. Math. \textbf{81} (1985), no.~3, 515--538. \MR{807070
  (87g:11073)}

\bibitem[JS81a]{Jb}
H.~Jacquet and J.~A. Shalika, \emph{On {E}uler products and the classification
  of automorphic representations. {I}}, Amer. J. Math. \textbf{103} (1981),
  no.~3.

\bibitem[JS81b]{Ja}
\bysame, \emph{On {E}uler products and the classification of automorphic
  representations. {II}}, Amer. J. Math. \textbf{103} (1981), no.~4, 777--815.

\bibitem[Kim03]{K}
Henry~H. Kim, \emph{Functoriality for the exterior square of {${\rm GL}_4$} and
  the symmetric fourth of {${\rm GL}_2$}}, J. Amer. Math. Soc. \textbf{16}
  (2003), no.~1, 139--183 (electronic), With appendix 1 by Dinakar Ramakrishnan
  and appendix 2 by Kim and Peter Sarnak. \MR{1937203 (2003k:11083)}

\bibitem[Kim04]{KI}
\bysame, \emph{An example of non-normal quintic automorphic induction and
  modularity of symmetric powers of cusp forms of icosahedral type}, Invent.
  Math. \textbf{156} (2004), no.~3, 495--502. \MR{2061327 (2005f:11101)}

\bibitem[KS02]{KS}
Henry~H. Kim and Freydoon Shahidi, \emph{Cuspidality of symmetric powers with
  applications}, Duke Math. J. \textbf{112} (2002), no.~1, 177--197.
  \MR{1890650 (2003a:11057)}

\bibitem[KW09]{Khare}
Chandrashekhar Khare and Jean-Pierre Wintenberger, \emph{Serre's modularity
  conjecture. {I}}, Invent. Math. \textbf{178} (2009), no.~3, 485--504.
  \MR{2551763 (2010k:11087)}

\bibitem[Lan80]{GL}
Robert~P. Langlands, \emph{Base change for {${\rm GL}(2)$}}, Annals of
  Mathematics Studies, vol.~96, Princeton University Press, Princeton, N.J.,
  1980. \MR{574808 (82a:10032)}

\bibitem[PR11]{RP}
Dipendra Prasad and Dinakar Ramakrishnan, \emph{On the cuspidality criterion
  for the {A}sai transfer to {${\rm GL}(4)$}}, preprint, 2011.

\bibitem[Ram02]{Asai}
Dinakar Ramakrishnan, \emph{Modularity of solvable {A}rtin representations of
  {${\rm GO}(4)$}-type}, Int. Math. Res. Not. (2002), no.~1, 1--54. \MR{1874921
  (2003b:11049)}

\bibitem[Sas11a]{Sasaki1}
Shu Sasaki, \emph{On {A}rtin representations and nearly ordinary {H}ecke
  algebras over totally real fields. {I}}, preprint, 2011.

\bibitem[Sas11b]{Sasaki2}
\bysame, \emph{On {A}rtin representations and nearly ordinary {H}ecke algebras
  over totally real fields. {I}{I}}, preprint, 2011.

\bibitem[Ser77]{Serre}
J.-P. Serre, \emph{Modular forms of weight one and {G}alois representations},
  Algebraic number fields: {$L$}-functions and {G}alois properties ({P}roc.
  {S}ympos., {U}niv. {D}urham, {D}urham, 1975), Academic Press, London, 1977,
  pp.~193--268.

\bibitem[SPDyD94]{Diaz}
A.~Schwarz, M.~Pohst, and F.~Diaz~y Diaz, \emph{A table of quintic number
  fields}, Math. Comp. \textbf{63} (1994), no.~207, 361--376. \MR{1219705
  (94i:11108)}

\bibitem[Tun81]{Tunnell}
Jerrold Tunnell, \emph{Artin's conjecture for representations of octahedral
  type}, Bull. Amer. Math. Soc. (N.S.) \textbf{5} (1981), no.~2, 173--175.

\end{thebibliography}
\end{document}